\documentclass[12pt]{article}
\usepackage{amsmath,amsthm,amssymb,enumerate,xspace}

\newtheorem{theorem}{Theorem}[section]
\newtheorem{defi}[theorem]{Definition}
\newtheorem{remark}[theorem]{Remark}
\newtheorem{lemma}[theorem]{Lemma}

\newtheorem{prop}[theorem]{Proposition}
\newtheorem{cor}[theorem]{Corollary}

\newcommand{\tp}{\mathrm{tp}}
\newcommand{\qftp}{\mathrm{qftp}}

\newcommand{\Age}{\mathrm{Age}}

\newcommand{\K}{\mathcal{K}}
\newcommand{\AP}{$\mathrm{AP}$\xspace}
\newcommand{\JEP}{$\mathrm{JEP}$\xspace}
\newcommand{\HP}{$\mathrm{HP}$\xspace}
\newcommand{\Ka}{\K^\ast}
\newcommand{\To}{T^0}
\newcommand{\Tp}{T^\ast_{>0}}
\newcommand{\Lp}{L^\ast_{>0}}
\newcommand{\eq}{^\mathrm{eq}}
\newcommand{\wt}{\widetilde}
\newcommand{\tree}{2^{<\omega}}

\begin{document}
\title{On many-sorted $\omega$-categorical theories}

\author{Enrique Casanovas, Rodrigo Pel\'aez, and Martin
  Ziegler\thanks{Work partially supported by MODNET, FP6 Marie Curie
    Training Network in Model Theory and its Applications with
    contract number MRTN-CT-2004-512234. The first author has been
    partially supported by grants MTM 2008-01545 and 2009SGR-187.}}
\date{March 17, 2011} \maketitle

\begin{abstract} We prove that every many-sorted $\omega$-categorical
theory is completely interpretable in a one-sorted
$\omega$-categorical theory. As an application, we give a short proof
of the existence of non $G$--compact $\omega$-categorical theories.
\end{abstract}

\section{Introduction}
A many--sorted structure can be easily transformed into a one--sorted
by adding new unary predicates for the different sorts. However
$\omega$-categoricity is not preserved.  In this article we present a
general method for producing $\omega$-categorical one--sorted
structures from $\omega$-categorical many--sorted structures.  This is
stated in Corollary~\ref{interpret}, the main theorem in this paper.
Our initial motivation was to understand Alexandre Ivanov's example
(in~\cite{Ivanov06}) of an $\omega$-categorical non $G$-compact
theory.  In Corollary~\ref{Ivanov} we apply our results to offer a
short proof of the existence of such theories.

Our method is based on the use of a particular theory $T_E$ of
equivalence relations $E_n$ on $n$-tuples. The quotient by $E_n$ is an
imaginary sort containing a predicate $P_n$ which can be used to copy
the $n$-th sort of the given many-sorted theory. Since the complexity
of $T_E$ is part of the complexity of the $\omega$-categorical
one--sorted theory obtained by our method, it is important to classify
$T_E$ from the point of view of stability, simplicity and related
properties. It turns out that $T_E$ is non-simple but it does not have
$\mathrm{SOP}_2$. A similar example of a theory with such properties
has been presented by Shelah and Usvyatsov
in~\cite{ShelaUsvyatsov03}. Their proof, as ours, relies on Claim 2.11 of~\cite{DzamonjaShelah04}, which is
known to have some gaps.  A revised version of~\cite{DzamonjaShelah04} will be posted in arxiv.org. In the meanwhile Kim and Kim have obtained a new proof of the same result: Proposition 2.3 from~\cite{KimKim11}.

The one--sorted theory $T_E$ is interdefinable with some many--sorted
theory $T^\ast$ which is presented and discussed in
Section~\ref{section:Tast}. In order to describe $T^\ast$ we need a
version of Fra\"{\i}ss\'e's amalgamation method that can be applied to
the many--sorted case (see Lemma~\ref{F1}). In
Section~\ref{section:stable_embedded} some results on stable
embeddedness from the third author (in~\cite{Ziegler06}) are extended
and used to prove Corollary~\ref{interpret}.
Section~\ref{section:classification} is devoted to classify $T_E$ from
the stability point of view.

A previous version of these results appeared in the second author's
Ph.D. Dissertation \cite{pelaez_diss}.  They have been corrected in
some points and in general they have been elaborated and made more
compact.

\section{$T^\ast$ and Fra\"{\i}ss\'e's amalgamation}
\label{section:Tast}

Let $L$ be a countable many--sorted language with sorts $S_i$, ($i\in
I$), and let $\K$ be a class of finitely generated
$L$--structures.\footnote{We allow empty sorts if $L$ has no constant
  symbols of that sort.}  We call an $L$--structure $M$ a
\emph{Fra\"{\i}ss\'e limit} of $\K$ if the following holds:
\begin{enumerate}
\item $\K=\Age(M)$, where $\Age(M)$ is the class of all finitely
  generated $L$--structures which are embeddable in $M$.
\item $M$ is at most countable.
\item $M$ is \emph{ultra--homogeneous}\, i.e., any isomorphism between
  finitely generated substructures extends to an automorphism of $M$.
\end{enumerate}
By a well--known argument $\K$ can only have one Fra\"{\i}ss\'e limit,
up to isomorphism.

\begin{lemma}\label{F1} Let $\K$ be as above.
  Then the following are equivalent:
  \begin{enumerate}[a)]
  \item The Fra\"{\i}ss\'e limit of $\K$ exists and is
    $\omega$-categorical.
  \item $\K$ has the amalgamation property \AP, the joint embedding
    property \JEP, the hereditary property \HP (i.e., finitely
    generated $L$--structures which are embeddable in elements of $\K$
    belong themselves to $\K$) and satisfies
    \begin{itemize}
    \item[$(\ast)$] for all $i_1\ldots i_n \in I$ there are only
      finitely many quantifier-free types of tuples $(a_1,\ldots,a_n)$
      where the $a_j$ are elements of sort $S_{i_j}$ in some structure
      $A\in\K$.
    \end{itemize}
  \end{enumerate}
  If the Fra\"{\i}ss\'e limit of $\K$ exists, it has quantifier
  elimination.
\end{lemma}
\begin{proof}
  a) $\Rightarrow$ b). It is well known that the age of an
  ultra--homogeneous structure has \AP, \JEP and \HP. All
  quantifier--free types which occur in elements of $\K$ are
  quantifier--free types of tuples of the Fra\"{\i}ss\'e limit. So
  property $(\ast)$ follows from the Ryll--Nardzewski theorem.\\

  \noindent b) $\Rightarrow$ a). The quantifier-free type
  $\qftp(\bar{a})$ determines the isomorphism type of the structure
  generated by $\bar{a}$. Hence $(\ast)$ implies that $\K$ contains at
  most countably many isomorphism types. The existence of the
  Fra\"{\i}ss\'e limit $M$ follows now from \AP, \JEP and \HP.

  If two sequences $\bar{a}$ and $\bar{b}$ have the same
  quantifier-free type in $M$, there is an automorphism of $M$ which
  maps $\bar{a}$ to $\bar{b}$ and so it follows that $\bar{a}$ and
  $\bar{b}$ have the same type in $M$. Consider a formula
  $\varphi(\bar x)$ and the set
  $P_{\varphi(x)}=\{\qftp(\bar{a}): M\models \varphi(\bar{a})\}$. Then
  \[M\models\varphi(\bar{a}) \Leftrightarrow \qftp(\bar{a})\in P_\varphi
  \Leftrightarrow M\models \bigvee_{p\in P_\varphi} p(\bar{a}).\] Now,
  $(\ast)$ implies that $P_\varphi$ is finite and that in $M$ all
  $p=\qftp(\bar{a})$ are finitely axiomatisable, that is, $p=\langle
  \chi_p\rangle$ for some quantifier-free $\chi_p(x)$. Then $M\models
  \varphi(\bar{a}) \Leftrightarrow \bigvee_{p\in P_\varphi}\chi_p(\bar
  a)$. So $M$ has quantifier elimination and it is
  $\omega$--categorical since there are only finitely many
  possibilities for the $\chi_p$, depending only on the number and the
  sorts of the free variables of $\varphi$.
\end{proof}

It is easy to see that the theory of the Fra\"{\i}ss\'e limit is the
model--completion of the universal theory of $\K$.

\begin{defi}\rm
  Let $L^\ast$ be the language with countably many sorts
  $S,S_1,\ldots$, function symbols $f_i:S^i\rightarrow S_i$, and
  constants $c_i\in S_i$ and $\To$ the theory of all
  $L^\ast$--structures $A$ with
  \[f_i(\bar a)=c_i \Leftrightarrow \bar a \text{ has some repetition}\]
  for all $\bar a\in S^i(A)$.  Furthermore let $\Ka$ be the class of
  all finitely generated models of $\To$.
\end{defi}

\begin{lemma} $\Ka$ satisfies the conditions of Lemma~\ref{F1}.
\end{lemma}
\begin{proof}
  The class of all models of $\To$ has \AP and \JEP and therefore also
  $\Ka$.  $(\ast)$ follows easily from the fact that
  $f_i(a_{m_1}\ldots,a_{m_i})=c_i$ for all $i>k$ and
  $\{a_{m_1}\ldots,a_{m_i}\}\subset\{a_1,\ldots,a_k\}$,
\end{proof}

We define $M^\ast$ to be the Fra\"{\i}ss\'e limit of $\Ka$ and
$T^\ast$ to be the complete theory of $M^\ast$. $T^\ast$ is the
model--completion of $\To$.\\

Recall the following definition from \cite{Chat-Hru99}:
\begin{defi}\rm
  Let $T$ be a complete theory and $P$ a $0$--definable predicate. $P$
  is called \emph{stably embedded} if every definable relation on $P$
  is definable with parameters from $P$.
\end{defi}

\noindent \textbf{Remarks}
\begin{enumerate}
\item For many--sorted structures with sorts $(S_i)_{i\in I}$ this
  generalises to the notion of a sequence $(P_i)_{i\in I}$ of
  $0$--definable $P_i\subset S_i$ being stably embedded.
\item While the definition is meant in the monster model, an easy
  compactness argument shows that, if $P(M)$ is stably embedded in $M$
  for some weakly saturated\footnote{$M$ is weakly saturated if every
    type over the empty set is realized in $M$.} model $M$, then this
  is true for all models.
\item If $M$ is saturated then $P$ is stably embedded if and only if
  every automorphism (i.e.\ elementary permutation) of $P(M)$ extends
  to an automorphism of $M$. This was claimed in \cite{Chat-Hru99}
  only for the case that $|M|>|T|$. But the proof can easily be
  modified to work for the general case. One has to use the fact that
  if $A$ has smaller size than $M$, then any type over a subset of
  $\mathrm{dcl}\eq(A)$ can be realized in $M$.
\item If $M$ is $\omega$--categorical, it can be proved that for every
  finite tuple $a\in M$ there is a finite tuple $b\in P$ such that
  every relation on $P$ which is definable over $a$ can be defined
  using the parameter $b$.
\end{enumerate}

\begin{lemma} In $T^\ast$ the sequence of sorts $(S_1,S_2,\ldots)$ is
  stably embedded.
\end{lemma}
\begin{proof} Clear since $\tp(\bar{a}/S_1,\ldots)=
  \tp(\bar{a}/f_1(\bar{a}),\ldots)$.  See also the discussion in
  \cite{Chat-Hru99}.
\end{proof}

For a complete theory $T$ and a $0$--definable predicate $P$ the
\emph{induced structure} on $P$ consists of all $0$--definable
relations on $P$. Note that the automorphisms of $P$ with its induced
structure are exactly the elementary permutations of $P$ in the sense
of $T$.

\begin{lemma} In $T^\ast$ the induced structure on
  $(S_1,S_2,\ldots)$ equals its $\Lp$--structure, where $\Lp$ is the
  sublanguage of $L^\ast$ which has only the sorts\/ $S_1,S_2,\ldots$
  and the constants $c_1,c_2,\ldots$.
\end{lemma}
\begin{proof} Quantifier elimination.
\end{proof}

Let $\Tp$ denote the theory of all $\Lp$--structures, where all sorts
$S_i$ are infinite. Clearly $\Tp$ is the restriction of $T^\ast$ to
$\Lp$.

\begin{lemma}
  Every model of $\Tp$ can be expanded to a model of $T^\ast$.
\end{lemma}
\begin{proof}
  It is easy to see that the following amalgamation property is true:
  \begin{itemize}
  \item[] Let $N$ be a model of $\To$ with infinite sorts
    $S_i(N)$. Let $A$ be a finitely generated substructure of $N$ and
    $B\in\Ka$ an extension of $A$. Then $B$ can be embedded over $A$
    in an extension $N'$ of $N$ which is a model of $\To$ and such
    that $S_i(N')=S_i(N)$ for all $i$.
  \end{itemize}
  If a model of $\Tp$ is given, we expand it arbitrarily to a model
  $N$ of $\To$ and apply the above amalgamation property repeatedly
  such that the union of the resulting chain is a model of $T^\ast$
  which has the same sorts $S_i$ as $N$.
\end{proof}

\begin{cor} There is an $\omega$--categorical one-sorted theory $T_E$
  with a series of $0$--definable infinite predicates $P_1,P_2,\ldots$
  in $T_E\eq$ such that
  \begin{enumerate}
  \item $(P_1,P_2,\ldots)$ is stably embedded
  \item The many--sorted structure induced on $(P_1,P_2,\ldots)$ is
    trivial.
  \item For every sequence $\kappa_1,\kappa_2,\ldots$ of infinite
    cardinals there is a model $N$ of $T_E$ such that
    $|P_i(N)|=\kappa_i$.
  \end{enumerate}
\end{cor}
\begin{proof}
  The language $L_E$ of $T_E$ will contain for each $i$ a symbol $E_i$
  for an equivalence relation between $i$--tuples.  Let
  $M=(S,S_1,S_2,\ldots)$ be a model of $T^\ast$. For $a,b\in S^i$
  define $E_i(a,b)\Leftrightarrow f_i(a)=f_i(b)$. $T_E$ is the theory
  of $M_E=(S,E_1,E_2,\ldots)$. The $S_i$ live in $M_E\eq$ as
  $M_E^i/E_i$ and the $c_i$ are $0$--definable in $M_E\eq$. We set
  $P_i=S_i\setminus\{c_i\}$.
\end{proof}

It is easy to see that $T_E$ as constructed in the proof is the
model--completion of the theory of all structures $(M,E_1,E_2,\ldots)$
where $E_n$ is an equivalence relation on $M^n$ where one equivalence
class consists of all $n$--tuples which contain a repetition. That
$T_E$ has quantifier elimination can be proved as follows: Every
formula $\varphi(\bar x)$ of $L_E$ is equivalent to a quantifier--free
$L^\ast$--formula $\varphi'(\bar x)$. $\varphi'(\bar x)$ is a boolean
combination of formulas of the form $f_i(\bar x')\doteq f_i(\bar x'')$
and $f_i(\bar x')\doteq c_i$, which are equivalent to quantifier--free
$L_E$--formulas: $f_i(\bar x')\doteq f_i(\bar x'')$ is equivalent to
$E_i(\bar x',\bar x'')$, $f_i(x'_1,\ldots,x'_i)\doteq c_i$ is
equivalent to $\bigvee_{1\leq k<l\leq i}x'_k\doteq x'_l$.

\section{Expansions of stably embedded predicates}
\label{section:stable_embedded}

Let $T$ be complete theory with two sorts $S_0$ and $S_1$. We consider
$S_1$ as a structure of its own carrying the structure induced from
$T$ and denote by $T\restriction S_1$ the theory of $S_1$.

\begin{lemma}\label{martinlemma}
  Let $T$ be complete theory with two sorts $S_0$ and $S_1$.  Let
  $\wt T_1$ be a complete expansion of $T\restriction
  S_1$. Assume that $S_1$ is stably embedded. Then we have
  \begin{enumerate}
  \item\label{complete} $\wt T=T\cup\wt T_1$ is
    complete.\footnote{Actually we have:
      $S_1$ is stably embedded if and only if $\wt T$ is complete for
      all complete expansions $\wt T_1$.}
    (\cite[Lemma
    3.1]{Ziegler06})
  \item\label{stably} $S_1$ is stably embedded in $\wt T$ and $\wt
    T\restriction S_1=\wt T_1$.
  \item\label{omega} If\/ $T$ and $\wt T_1$ are $\omega$--categorical,
    then $\wt T$ is also $\omega$--categorical.
  \end{enumerate}
\end{lemma}
\begin{proof}
  \ref{complete}.  Let $\wt M=(M_0,\wt M_1)$ and $\wt M'=(M'_0,\wt
  M'_1)$ be saturated models of $\wt T$ of the same cardinality and
  $M=(M_0,M_1)$ and $M'=(M'_0,M'_1)$ their restrictions to the
  language of $T$.  Since $T$ and $\wt T_1$ are complete, there are
  isomorphisms $f:M\to M'$ and $g:\wt M_1\to \wt M'_1$. $gf^{-1}$ is
  an automorphism of $M'_1$. Since $M'_1$ is stably embedded in $M'$,
  $gf^{-1}$ extends to an automorphism $h$ of $M'$. $hf$ is now an
  isomorphism from $M$ to $M'$ which extends $g$.\\

  \noindent \ref{stably}. We use the same notation as in the proof of
  \ref{complete}. Let $\wt M$ be a saturated model of $\wt T$. We have
  to show that every automorphism $f$ of $\wt M_1$ extends to an
  automorphism of $\wt M$. But $f$ extends to an automorphism of $M$,
  which is automatically an automorphism of $\wt M$.\\

  \noindent\ref{omega}. Start with two countable models $\wt M$ and
  $\wt M'$ and proceed as in the proof of \ref{complete}.
\end{proof}

\begin{cor}\label{interpret}
  Every many-sorted $\omega$-categorical theory is completely (the
  induced structure is exactly this) interpretable in a one-sorted
  $\omega$-categorical theory.
\end{cor}
\begin{proof}
  Let $T$ be a complete theory with countably many sorts
  $P_1,P_2,\ldots$. We consider $T$ as an expansion of
  $T_E\restriction (P_1,P_2,\ldots)$ and set $\wt T=T_E\cup T$.  $\wt
  T$ is a one-sorted complete theory. We have $\wt T\restriction
  (P_1,P_2,\ldots)= T$. If $T$ is $\omega$--categorical, $\wt T$ is
  also $\omega$--categorical.
\end{proof}

\begin{cor}[Ivanov]\label{Ivanov}
  There is a one--sorted $\omega$--categorical theory which is not
  G--compact.
\end{cor}
\begin{proof}
  By \cite{CasanovasLascarPillayZiegler01} there is a many--sorted
  $\omega$--categorical theory $T$ which is not
  $G$--compact. Interpret $T$ in a one--sorted $\omega$--categorical
  theory $\wt T$ as in Corollary \ref{interpret}. Then $T$ is also not
  $G$--compact. For this one has to check that if $\wt T$ is
  $G$--compact, then every $0$--definable subset with its induced
  structure is also $G$--compact. This follows from the following
  description of $G$--compactness: $a$, $b$ of length $\omega$ are in
  the relation $\mathrm{nc^\omega}$ if $a$ and $b$ are the first two
  elements of an infinite sequence of indiscernibles. A complete
  theory is G--compact, if the transitive closure of
  $\mathrm{nc^\omega}$ is type--definable. (Note that $(a,b)$ is in the
  transitive closure of $\mathrm{nc^\omega}$ if and only if $a$ and $b$
  have the same Lascar-strong type.)
\end{proof}

\section{Classification of $T_E$}
\label{section:classification}

\begin{prop} $T_E$ has $\mathrm{TP}_2$, the tree
  property of the second kind, and therefore it is not simple.
\end{prop}
\begin{proof} We show that $\varphi(x;y,u,v)= E_2(xy,uv)$
  has $\mathrm{TP}_2$.  Let $(b_i:i < \omega)$, $(c_i:i < \omega)$,
  and $(d_i:i < \omega)$ be pairwise disjoint sequences of different
  elements such that $ \neg E_2(c_id_i,c_jd_j)$ for $i\neq j$. For
  $i,j \in \omega$, let $\bar a^i_j=b_ic_jd_j $. By compactness we can
  see that for any $\eta \in \omega^{\omega}$, the set
  $\{\varphi(x;\bar a^i_{\eta(i)}): i < \omega\}$ is consistent, and
  since the $c_id_i$'s are in different $E_2$-classes, for each $i <
  \omega$, the set $\{\varphi(x; \bar a^i_j): j < \omega\}$ is
  $2$-inconsistent.
\end{proof}

\begin{lemma}[Independence lemma] \label{N2}
  Let $a,b,c,d^\prime,d^{\prime\prime}$ be tuples in the monster model
  of $T_E$ and $F$ a finite subset. Assume that $a$ and $c $ have only
  elements from $F$ in common. If $d^\prime a\equiv_F d^\prime
  b\equiv_F d^{\prime\prime}b\equiv_F d^{\prime\prime}c$, then there
  exists some $d$ such that $d^\prime a\equiv_F da\equiv_F dc\equiv_F
  d^{\prime\prime} c$.
\end{lemma}

\begin{center}
\setlength{\unitlength}{1.8cm}
\begin{picture}(2,1.5)(-1,-0.25)
\put(-1,0){$a$}
\put(-.95,0.2){\vector(0,1){.7}}
\put(-1,1){$d^\prime$}
\put(-.1,1.05){\vector(-1,0){.7}}
\put(0,1){$b$}
\put(0.2,1.05){\vector(1,0){.7}}
\put(1,1){$d^{\prime\prime}$}
\put(1.05,.2){\vector(0,1){.7}}
\put(1,0){$c$}
\put(.9,0.05){\vector(-1,0){.7}}
\put(0,0){$d$}
\put(-.8,.05){\vector(1,0){.7}}
\end{picture}
\end{center}

\begin{proof} Let $A$, $B$, $C$, $D'$ and $D''$ denote the
  set of elements of the tuples $a$, $b$, $c$, $d'$ and $d''$,
  respectively. We note first that we can assume that $F$ is contained
  in $A$,$B$ and $C$, since otherwise we can increase $a$, $b$ and $c$
  by elements from $F$.  Then we note that if $A$ and $D'$ intersect
  in a subtuple $f$, this tuple also belongs to $B$ and $C$ and
  therefore to $F$. So we have that $A\cap D'$ is contained in $F$ and
  similarly that $C\cap D''$ is contained in $F$.

  It suffices to find an $L_E$-structure $M$ extending $AC$ and
  containing a new tuple $d$ with the same quantifier-free type as
  $d^\prime$ over $A$ and of $d^{\prime\prime}$ over $C$.  Take as $d$
  a new tuple of the right length which intersects $A$ and $B$ in the
  subtuple $f$. We have then $d^\prime
  a\equiv_F^\mathrm{eq}da\equiv_F^\mathrm{eq}dc\equiv_F^\mathrm{eq}
  d^{\prime\prime} c$, where $g\equiv_F^\mathrm{eq}h$ means that $g$
  and $h$ satisfy the same equality-formulas over $F$, i.e.\ $g_i=g_j$
  iff $h_i=h_j$ and $g_i=f_j$ iff $h_i=f_j$. If $D$ denotes the
  elements of $d$, it follows that the intersection of any two of $A$,
  $C$ and $D$ belongs to $F$.

  It remains to define the relations $E_n$ on $ACD$. Let $E_n^0$
  denote the part of $E_n$ which is already defined on $AC$. Let
  $E_n'$ be the relation $E_n$ transported from $AD'$ to $AD$ via the
  identification $d'\mapsto d$ and $E_n''$ the relation $E_n$
  transported from $CD''$ to $CD$ via the identification $d''\mapsto
  d$.  Note that $d'\equiv_F d''$ implies that $E_n'$ and $E_n''$
  agree on $DF$.  We define $E_n$ on $ACD$ as the transitive closure
  of
  \[E_n^0\cup E_n'\cup E_n''\cup E_n^\mathrm{rep}\cup\Delta,\]
  where $E_n^\mathrm{rep}$ is the set of all pairs of $n$--tuples from
  $ACD$ which contain repetitions and $\Delta$ is the identity on
  $(ACD)^n$.

  We have to show that the new structure defined on $AC$ agrees with
  the original structure.  Also we must check that the structure on
  $AD$ (and $CD$) agrees with the structure on $AD'$ (and $CD'$) via
  $d\mapsto d'$ (and $d\mapsto d''$).  Using the fact that an
  $n$-tuple which e.g.\ belongs to $AC$ and $AD$ belongs already to
  $A$, it is easy to see that we have to show the following:
  \begin{itemize}
  \item[] For all $n$--tuples $x\in A$, $y\in C$ and $z\in DF$
    \begin{enumerate}
    \item\label{xzy} $E_n'(x,z)\land
      E_n''(z,y)\;\Rightarrow\;E_n^0(y,x)$
    \item\label{zyx} $E_n''(z,y)\land
      E_n^0(y,x)\;\Rightarrow\;E_n'(x,z)$
    \item\label{yxz} $E_n^0(y,x)\land
      E_n'(x,z)\;\Rightarrow\;E_n''(z,y)$
    \end{enumerate}
  \end{itemize}
  Let $z'$ and $z''$ be the subtuples of $D'F$ and $D''F$ which
  correspond to $z$.\\

  \noindent Proof of \ref{xzy} : Assume $E_n'(x,z)$ and
  $E_n''(z,y)$. We have then $E_n(x,z')$ and $E_n(z'',y)$. $d^\prime
  a\equiv_F d^\prime b$ implies $z'a\equiv_F z'b$, which implies that
  there is a tuple $x'$ in $B$ such that $z'x\equiv_F z'x'$. So we
  have $E_n(z',x')$. $d^\prime b\equiv_F d^{\prime\prime}b$ implies
  $z'x'\equiv_F z''x'$ and whence $E_n(z'',x')$. Now we can connect
  $y$ and $x$ as follows: $y\;E_n\;z''\;E_n\;x'\;E_n\;z'\;E_n\;x$.\\

  \noindent Proof of \ref{zyx} : Assume $E_n''(z,y)$ and
  $E_n^0(y,x)$. We have then $E_n(z'',y)$. As above we find
  a tuple $y'\in B$ such that $E_n(z'',y')$ and $E_n(z',y')$.
  The chain $x\;E_n\;y\;E_n\;z''\;E_n\;y'\;E_n\;z'$ shows that
  $E_n'(x,z)$.\\

  \noindent Proof of \ref{yxz} : Symmetrical to the proof of
  \ref{zyx}.
\end{proof}

In order to state \cite[Proposition 2.3]{KimKim11} we
need the following terminology:
\begin{enumerate}[(1)]
\item A tuple $\bar\eta=(\eta_0,\ldots,\eta_{d-1})$ of
  elements of $\tree$ is \emph{$\cap$-closed} if the set
  $\{\eta_0,\ldots,\eta_{d-1}\}$ is closed unter intersection.
\item Two $\cap$-closed tuples $\bar\eta$ and $\bar\nu$ are
  \emph{isomorphic} if they have the same length and
  \begin{enumerate}[(i)]
  \item $\eta_i\unlhd\eta_j$ iff $\nu_i\unlhd\nu_j$
  \item $\eta_i^\smallfrown t\unlhd\eta_j$ iff $\nu_i^\smallfrown
    t\unlhd\nu_j$ for $t=0,1$.
  \end{enumerate}
\item A tree $(a_\eta\colon\eta\in\tree)$ of tuples of the same length
  is \emph{modeled} by $(b_\eta\colon\eta\in\tree)$ if for every
  formula $\phi(\bar x)$ and every $\cap$-closed $\bar\eta$ there is a
  $\cap$--closed $\bar\nu$ isomorphic to $\bar\eta$ such that
  $\models\phi(b_{\bar\eta})\,\Leftrightarrow\,\models\phi(a_{\bar\nu})$.
\item $(b_\eta\colon\eta\in\tree)$ is \emph{indiscernible} if
  $\models\phi(b_{\bar\eta})\,\Leftrightarrow\,\models\phi(b_{\bar\nu})$
  for all isomorphic $\cap$-closed $\bar\eta,\bar\nu$.
\end{enumerate}

\begin{lemma}[\protect{\cite[Proposition 2.3]{KimKim11}. See also~\cite{DzamonjaShelah04}}]
  \label{N1}
  Let $T$ be a complete theory. Then any tree of tuples can be modeled
  by an indiscernible tree.
\end{lemma}

\begin{defi}\rm The formula $\varphi(x,y)$ has
  $\mathrm{SOP}_2$ in $T$ if there is a binary tree $(a_\eta: \eta\in
  2^{<\omega})$ such that for every $\eta\in 2^\omega$,
  $\{\varphi(x,a_{\eta\restriction n}): n<\omega\}$ is consistent and
  for every incomparable $\eta, \nu\in 2^{<\omega}$,
  $\varphi(x,a_\eta) \wedge \varphi(x,a_\nu)$ is inconsistent.  The
  theory $T$ has $\mathrm{SOP}_2$ if some formula $\varphi(x,y)\in L$
  has $\mathrm{SOP}_2$ in $T$.
\end{defi}

\begin{remark}[H. Adler] The formula $\varphi(x,y)$ has
$\mathrm{SOP}_2$ in $T$ if and only if $\varphi(x,y)$ has the tree
  property of the first kind $\mathrm{TP}_1$: there is a tree
  $(a_\eta: \eta\in \omega^{<\omega})$ such that for every $\eta\in
  \omega^\omega$, $\{\varphi(x,a_{\eta\restriction n}): n<\omega\}$ is
  consistent and for every incomparable $\eta, \nu\in
  \omega^{<\omega}$, $\varphi(x,a_\eta) \wedge \varphi(x,a_\nu)$ is
  inconsistent.
\end{remark}
\begin{proof} By compactness.
\end{proof}

\begin{prop} $T_E$ does not have $\mathrm{SOP}_2$.
\end{prop}
\begin{proof} We follow ideas from a similar proof
  in~\cite{ShelaUsvyatsov03}. Assume $\varphi(x,y)$ has
  $\mathrm{SOP}_2$ in $T_E$ and the tree $(a_\eta: \eta\in
  2^{<\omega})$ witnesses it. Choose for every $\eta$ a tuple $d_\eta$
  such that $\models\phi(d_\eta,a_\nu)$ for all $\nu\subsetneq\eta$.

  By Lemma~\ref{N1} we can assume that the tree $(d_\eta a_\eta :
  \eta\in 2^{<\omega})$ is indiscernible. Let us now look at the
  elements $a_{00},a_{\langle\rangle},a_{01},d_{000}, d_{010}$.  We
  have by indiscernibility
  \[d_{000}a_{00}\equiv d_{000}a_{\langle\rangle}\equiv
  d_{010}a_{\langle\rangle}\equiv d_{010}a_{01}.\]

  If the tuples $a_{00}$ and $a_{01}$ are disjoint, we can apply the
  Independence Lemma to $a=a_{00}$, $b=a_{\langle\rangle}$,
  $c=a_{01}$, $d'=d_{000}$, $d''=d_{010}$ to get a tuple $d$ such that
  \[d_{000}a_{00}\equiv da_{00} \equiv
  da_{01}\equiv d_{010}a_{01}.\] It follows that
  $\models\varphi(d,a_{00})\wedge\varphi(d,a_{01})$, which contradicts
  the $\mathrm{SOP}_2$ of the tree.

  If $a_{00}$ and $a_{01}$ are not disjoint, we argue as follows:
  Assume that $a_{00}$ and $a_{01}$ have an element $f$ in common, say
  $f=a_{00,i}=a_{01,j}$. Then $a_{00}a_{01}\equiv a_{000}a_{01}$
  implies $a_{000,i}=a_{01,j}$. So we have $a_{000,i}=a_{00,i}$ and it
  follows from indiscernibility that
  $f=a_{00,i}=a_{\langle\rangle,i}=a_{01,i}$. Let $F$ be the set of
  elements which occur in both $a_{00}$ and $a_{01}$. We have seen
  that the elements of $F$ occur in $a_{00}$, $a_{\langle\rangle}$ and
  $a_{01}$ at the same places. Therefore \[d_{000}a_{00}\equiv_F
  d_{000}a_{\langle\rangle}\equiv_F d_{010}a_{\langle\rangle}\equiv_F
  d_{010}a_{01}\] and we can again apply the Independence Lemma.

\end{proof}

\nocite{Ivanov06}\nocite{Chat-Hru99}\nocite{Ziegler06}
\nocite{CasanovasLascarPillayZiegler01}

\bibliography{CPZ}
\bibliographystyle{plain}

\noindent{\sc
Department of Logic, History and Philosopy of Science,
University of Barcelona,
Montalegre 5, 08001 Barcelona, Spain.}\\
{\tt e.casanovas@ub.edu}\\

\noindent{\sc
Department of Logic, History and Philosopy of Science,
University of Barcelona,
Montalegre 5, 08001 Barcelona, Spain.}\\
{\tt rpelaezpelaez@yahoo.com}\\

\noindent{\sc Mathematisches Institut, Albert-Ludwigs-Universit\"at
  Freiburg, D-79104 Freiburg, Germany}\\ {\tt ziegler@uni-freiburg.de}

\end{document}